\documentclass[12pt]{amsart}

\usepackage[margin=1.15in]{geometry}
\usepackage{amscd,amssymb, amsmath,amsfonts, wasysym, mathrsfs, mathtools,enumerate,hhline,color}
\usepackage{graphicx}
\usepackage[all, cmtip]{xy}
\usepackage{url}

\definecolor{hot}{RGB}{65,105,225}

\usepackage[pagebackref=true,colorlinks=true, linkcolor=hot ,  citecolor=hot, urlcolor=hot]{hyperref}%make all the references and links clickable

\theoremstyle{plain}
\newtheorem{theorem}{Theorem}[section]
\newtheorem{prop}[theorem]{Proposition}

\newtheorem{conj}[theorem]{Conjecture}
\newtheorem{lemma}[theorem]{Lemma}

\theoremstyle{definition}

\newtheorem{defn}[theorem]{Definition}
\newtheorem{rmk}[theorem]{Remark}

\newtheorem*{ex*}{Example}

\newcommand\sF{{\mathcal F}}

\def\fg{\mathfrak{g}}

\newcommand\pp{{\mathbb{P}}}

\newcommand\zz{{\mathbb{Z}}}

\newcommand\cc{{\mathbb{C}}}

\newcommand\nn{{\mathbb{N}}}

\newcommand\hh{{\mathbb{H}}}

\DeclareMathOperator{\reg}{reg}                  % reg
\DeclareMathOperator{\sing}{sing}                  
\DeclareMathOperator{\gdeg}{gdeg}
\DeclareMathOperator{\MLdeg}{MLdeg}
\DeclareMathOperator{\IC}{IC}

\title[Maximum Likelihood degree]{Bounding the maximum likelihood degree}
\begin{document}
\author{Nero Budur}
\email{Nero.Budur@wis.kuleuven.be}
\address{KU Leuven and University of Notre Dame}
\curraddr{KU Leuven, Department of Mathematics,
Celestijnenlaan 200B, B-3001 Leuven, Belgium}

\author{Botong Wang}
\email{bwang3@nd.edu}
\address{University of Notre Dame}
\curraddr{Department of Mathematics,
 255 Hurley Hall, IN 46556, USA} 
%\date{}

\thanks{The first author was partly sponsored by the Simons Foundation, NSA, and a KU Leuven OT grant.
}

\begin{abstract}
Maximum likelihood estimation is a fundamental computational problem in statistics. In this note, we give a bound for the maximum likelihood degree of algebraic statistical models for discrete data. As usual, such models are identified with special very affine varieties. Using earlier work of Franecki and Kapranov, we prove that the maximum likelihood degree is always less or equal to the signed intersection-cohomology Euler characteristic. We construct counterexamples to a bound in terms of the usual Euler characteristic conjectured by Huh and Sturmfels. 
\end{abstract}

%\thanks{The first author was partly sponsored by the Simons Foundation, NSA, and a KU Leuven OT grant.}

\maketitle
\section{Maximum Likelihood degree}

Maximum likelihood estimation is a fundamental computational problem in statistics. In order to estimate the parameters of a statistical model, one tries to maximize a likelihood function. In this note, we address algebraic statistical models for discrete data. As usual, such models are identified with special very affine varieties, that is, with closed subvarieties of $(\cc^*)^n$ for which the coordinates of every point sum up to 1. The number of critical points of the likelihood function for generic data depends only on the model and it is called the maximum likelihood degree. In this note, we give a bound for the maximum likelihood degree. 

The maximum likelihood degree was introduced in \cite{CHKS}. We will follow directly the algebraic definition of \cite{Hu}. For the explanation of how one gets from statistics to algebra, see \cite{CHKS, Hu, HS}. Let $(p_1, \ldots, p_n)$ be the coordinates of the complex torus $(\cc^*)^n$. 

\begin{defn}
Let $X$ be a closed irreducible subvariety of $(\cc^*)^n$. Denote the regular locus and the singular locus of $X$ by $X_{\reg}$ and $X_{\sing}$, respectively. The \textbf{maximum likelihood degree} of $X$ is defined to be the number of points in $X_{\reg}$, where the 1-form $$\frac{\lambda_1dp_1}{p_1}+\cdots+\frac{\lambda_ndp_n}{p_n}$$ degenerates, for a generic point 
$(\lambda_1, \ldots, \lambda_n)\in \cc^{n}$. The maximum likelihood degree of $X$ is denoted by $\MLdeg(X)$. 
\end{defn}
%\subsection{Relation with the Euler characteristics}
The maximum likelihood degree is closely related to the Euler characteristic. Let $d$ be the dimension of $X$. 
\begin{theorem}{\cite[Theorem 1]{Hu}}
When $X$ is smooth, 
$$
(-1)^d\chi(X)=\MLdeg(X). 
$$
\end{theorem}

Denote the hyperplane $p_1+\cdots+p_n=1$ in $(\cc^*)^n$ by $H$. In \cite{HS}, Huh and Sturmfels made the following conjecture for subvarieties of $(\cc^*)^n$ contained in $H$, that is, for special very affine varieties:

\begin{conj}{\cite[Conjecture 1.8]{HS}}\label{conjHS}
Suppose $X$ is a closed irreducible subvariety of $H$. Then 
$$
(-1)^d\chi(X)\geq \MLdeg(X).
$$
In particular, the topological signed Euler characteristic $(-1)^d\chi(X)$ is always nonnegative. 
\end{conj}

The main attractive feature of the conjecture is that it bounds an algebraic invariant, $\MLdeg(X)$, by something that requires much less than algebra, since $\chi(X)$ is a homotopy invariant.

Moreover, Huh and Sturmfels also conjectured that $(-1)^d\chi(V)\geq 0$, where $V$ is any irreducible closed subvariety of $(\cc^*)^m$ of dimension $d$. However, we constructed in \cite{BW} a family of surfaces in $(\cc^*)^4$ of arbitrary negative Euler characteristics. The failure of the usual Euler characteristic to provide a bound is due to the presence of non-mild singularities. 

It has been known, more generally, that the usual cohomology of a stratified topological space $X$ is badly behaved in the presence of singularities. As a remedy for this, Goresky and MacPherson introduced intersection cohomology, $IH^i(X, \cc)$. This agrees with the usual cohomology $H^i(X,\cc)$ when $X$ is smooth, but is different in general.  In practice, the intersection cohomology of $X$ is the hypercohomology of the so-called intersection complex $\IC(\cc_X)$, and one defines the intersection-cohomology Euler characteristic as:
$$
\chi(\IC(\cc_X))=\sum_i (-1)^i\dim \hh^i(X, \IC(\cc_X))=\sum_i (-1)^i\dim IH^i(X, \cc).
$$
Intersection-cohomology Euler characteristics are addressed from the point of view of Gaussian degrees in \cite{FK}.

By relating the maximum likelihood degree with the Gaussian degree of \cite{FK}, we show that the example in \cite{BW} leads to arbitrarily-bad counterexamples to Conjecture \ref{conjHS}. We also show that the conjecture is true if one replaces the signed Euler characteristic  by the signed Euler characteristic of the intersection complex:

\begin{theorem}\label{ic}
Let $X$ be an irreducible subvariety of $(\cc^*)^n$ of dimension $d$. Then
$$
(-1)^d\chi(\IC(\cc_{X}))\geq \MLdeg(X). 
$$
\end{theorem}

The main attractive feature here is that the algebraic invariant $\MLdeg(X)$ is bounded by something that requires much less than algebra, the stratified homotopy invariant $\chi(\IC(\cc_{X}))$.

The first author would like to thank  B. Caffo and C. Cr\u{a}iniceanu for several discussions. 

\section{Gaussian degree}
We start with the review of the definition of Gaussian degree and the results of \cite{FK}. At the end of the section, we prove Theorem \ref{ic}.

Denote the complex torus $(\cc^*)^{n}$ by $G$, and denote its Lie algebra by $\mathfrak{g}$. Let $T^*G$ be the cotangent bundle of $G$. $T^*G$ has a canonical symplectic structure. For $\gamma\in \fg^*$, let $\Omega_\gamma$ be the graph of the corresponding left invariant 1-form on $T^*G$. Suppose $\Lambda$ is a closed Lagrangian subvariety of $T^*G$. Consider the horizontal projection $\pi: T^*G\to \fg^*$, which maps the whole section $\Omega_\gamma\subset T^*G$ to $\gamma\in \fg^*$. More precisely, it maps $(x, \eta)\in T^*G$ to $\mathsf{T}_{x^{-1}}^*(\eta)$, where $\mathsf{T}_{x^{-1}}: G\to G$ is the translation map by $x^{-1}$. Since $\dim \Lambda=\dim \fg^*=n$, the map $\pi|_\Lambda: \Lambda\to \fg^*$ is generically finite. Since $\Lambda$ is a variety, $\pi_\Lambda$ is generically smooth. Therefore, for a generic point $\gamma\in \fg^*$, the intersection $\Lambda\cap \Omega_\gamma$ is transversal and consists of finitely many points, possibly empty. The number of the intersection points is called the Gaussian degree of $\Lambda$, and denoted by $\gdeg(\Lambda)$. Clearly, the Gaussian degree of $\Lambda$ is equal to the degree of the field extension $K(\fg^*)\to K(\Lambda)$ of the function fields induced by $\pi|_\Lambda$, which is a nonnegative integer. Here we allow the induced map $K(\fg^*)\to K(\Lambda)$ to be zero map, in which case the degree of the extension is zero. 

Let $V$ be a subvariety of $G$. Denote the conormal bundle of $V_{\reg}$ in $T^*G$ by $T_{V_{\reg}}^*G$ and denote its closure in $T^*G$ by $T_V^*G$. Then $T_{V}^*G$ is an irreducible conic Lagrangian subvariety of $T^*G$. When $\gamma\in \fg$ is generic, the intersection $T_V^*G\cap \Omega_\gamma$ is contained in $T_{V_{\reg}}^*G$. 

The 1-form $\gamma$ degenerates at some point $P\in V_{\reg}$ if and only if $T_{V_{\reg}}^*G\cap \Omega_{\gamma}$ contains a point in $T_P^*G$. Therefore, one immediately has the following lemma. 

\begin{lemma}\label{mlg}
Given any irreducible subvariety $X$ of $G=(\cc^*)^n$, 
$$
\MLdeg(X)=\gdeg(T_X^*G).
$$
That is, the maximum likelihood degree of $X$ is equal to the Gaussian degree of $T_X^*G$. 
\end{lemma}

The result of Franecki and Kapranov \cite{FK} relates Gaussian degree with Euler characteristics. Let $\sF$ be a bounded constructible complex on $G$ and let $CC(\sF)$ be its characteristic cycle (see \cite{KS} for the definition). Then $CC(\sF)=\sum_{j}n_j[\Lambda_j]$ is a $\zz$-linear combination of irreducible conic Lagrangian subvarieties in the cotangent bundle $T^*G$. The main result of \cite{FK} is the following\footnote{In \cite{FK}, the theorem is proved more generally for any semi-abelian variety $G$. }. 

\begin{theorem}{\cite[Theorem 1.3]{FK}}\label{Riemann}
Under the above notations, 
$$
\chi(G, \sF)=\sum_j n_j \cdot\gdeg(\Lambda_j). 
$$
\end{theorem}

\begin{proof}[Proof of Theorem \ref{ic}]
Since $X$ is of dimension $d$, $\IC(\cc_X[d])$ as a constructible complex on $G$ is a perverse sheaf. Let $CC(IC(\cc_X[d]))=\sum_j n_j [\Lambda_j]$. The definition of the characteristic cycles is local. Along the regular locus $X_{\reg}$, the characteristic cycle is the cotangent bundle $T_{X_{\reg}}^*G$. At the singular locus $X_{\sing}$, the characteristic cycle may contain more cycles. Therefore we can assume $\Lambda_0=T_{X}^*G$ and $n_0=1$. Since $\IC(\cc_X[d])$ is a perverse sheaf, all the coefficients $n_j$ are nonnegative. Recall that, by definition, $\gdeg(\Lambda)\geq 0$ for any Lagrangian subvariety $\Lambda\subset T^*G$.  Thus,
$$
\chi(G, \IC(\cc_X[d]))=\sum_j n_j \gdeg(\Lambda_j)\geq n_0 \gdeg(\Lambda_0)=\gdeg(T_X^*G)=\MLdeg(X)
$$
where the last equality is by Lemma \ref{mlg}. 
\end{proof}

\section{Counterexamples}
In \cite{BW}, the authors gave examples of irreducible 2-dimensional subvarieties of $(\cc^*)^4$ with negative Euler characteristics. %Based on these examples, we will construct counterexamples of Conjecture \ref{conjHS}. 
Let us review the construction of these examples first. 

We start with the smooth surface
$$
U=\{(p_1, p_2, p_3, p_4)\in (\cc^*)^4 \mid p_1+p_3=p_2+p_4=1\}
$$
in $(\cc^*)^4$. For any $m\in \nn$, we define a map of torus $\pi_m: (\cc^*)^4\to (\cc^*)^4$ by
$$
\pi_m: (p_1, p_2, p_3, p_4) \mapsto (p_1^m, \;\frac{p_1}{p_2},\; p_1p_3, \;p_1p_4). 
$$
Let $U_m=\pi_m(U)$. Since $\pi_m$ is finite and proper, $U_m$ are irreducible 2-dimensional subvarieties of $(\cc^*)^4$. 
\begin{prop}{\cite[Corollary 3.2]{BW}}\label{sing}
Assume $m$ is odd. Then $U_m$ has $\frac{m-1}{2}$ isolated singular points. Moreover, the germ of $U_m$ at each singular point is analytically equivalent to the germ of $\{p_1=p_2=0\}\cup \{p_3=p_4=0\}$ in $\cc^4$ at origin. In other words, locally the singularity is obtained by the transverse intersection of two smooth surfaces.
\end{prop}
It follows from the Proposition \ref{sing} that the Euler characteristic of $U_m$ is  $$\chi(U_m)=\frac{3-m}{2}.$$

Now, we construct counterexamples to Conjecture \ref{conjHS}. Denote by $H^n$ the hyperplanes $\{p_1+\cdots+p_n=1\}$ in $(\cc^*)^n$. Then there is an isomorphism $$\iota: (\cc^*)^4\setminus H^4\to H^5$$ given by
$$
\iota: (p_1, p_2, p_3, p_4)\mapsto (p_1, p_2, p_3, p_4, 1-p_1-p_2-p_3-p_4). 
$$
Let $V_m=\iota(U_m\setminus H^4)$. Then $V_m$ are irreducible 2-dimensional subvarieties of $H^5$. 

\begin{prop}\label{example}
Assume $m$ is odd. Then
$$\MLdeg(V_m)-\chi(V_m)=\frac{m-1}{2}.$$
\end{prop}
When $m$ is odd and larger than 1, the proposition provides a family of counter examples of Conjecture \ref{conjHS}. Before proving the proposition, we need the following lemma.
\begin{lemma}
When $m$ is odd, $V_m$ has $\frac{m-1}{2}$ isolated singular points. Moreover, analytically locally each singularity is obtained by the transverse intersection of two smooth surfaces. 
\end{lemma}
\begin{proof}[Proof of Lemma]
Since $\iota: (\cc^*)^4\setminus H^4\to H^5$ is an isomorphism, it suffices to prove that $U_m\setminus H^4$ has $\frac{m-1}{2}$ isolated singular points and analytically locally each singularity is obtained by the transverse intersection of two smooth surfaces. According to Proposition \ref{sing}, we only need to prove $H^4$ contains none of the singular points of $U_m$. 

Let $\xi=e^{\frac{2\pi \sqrt{-1}}{m}}$. It follows from the proof of \cite[Lemma 3.1 and Corollary 3.2]{BW} that the singular points of $U_m$ are precisely 
$$\left(\frac{1}{(1+\xi^i)^m}, 1, \frac{\xi^i}{(1+\xi^i)^2}, \frac{\xi^i}{(1+\xi^i)^2}\right)\in (\cc^*)^4$$
where $i\in \{1, \ldots, \frac{m-1}{2}\}$. 

Suppose $\frac{1}{(1+\xi^i)^m}+1+\frac{\xi^i}{(1+\xi^i)^2}+\frac{\xi^i}{(1+\xi^i)^2}=1$. Then $2\xi^i(1+\xi^i)^{m-2}+1=0$. According to Eisenstein's criterion, the polynomial $2x(1+x)^{m-2}+1$ is irreducible. Therefore $2x(1+x)^{m-2}+1$ and $x^{m-1}+x^{m-2}+\cdots+1$ can not share a common root. Thus, we have a contradiction to the equality $2\xi^i(1+\xi^i)^{m-2}+1=0$. 

%Since the argument of $\xi$ is $\frac{2\pi}{m}$ and the argument of $1+\xi^i$ is $\frac{i\pi}{m}$, the argument of $2\xi^i(1+\xi^i)^{m-2}$ is $0$, a contradiction to $2\xi^i(1+\xi^i)^{m-2}=-1$. Therefore, $\frac{1}{(1+\xi^i)^m}+1+\frac{\xi^i}{(1+\xi^i)^2}+\frac{\xi^i}{(1+\xi^i)^2}\neq 1$ for any $i\in \{1, \ldots, \frac{m-1}{2}\}$, and hence $H^4$ does not contain any of the singular points of $U_m$. 
\end{proof}
\begin{proof}[Proof of Proposition \ref{example}]
According to the preceding lemma, we can denote all the singular points of $V_m$ by $P_1, \ldots, P_{\frac{m-1}{2}}$. 

Since $\iota$ is an isomorphism, it follows from Proposition \ref{sing} that the germ of $V_m$ at each singular point is isomorphic to the transverse intersection of two smooth surfaces. Let $D\subset (\cc^*)^5$ be a small multi-disk centered at $P_1$. Then $V_m\cap D$ is the union of two smooth surfaces, which we denote by $S_1$ and $S_2$. By the definition of intersection complexes, 
$$
IC(\cc_{V_m\cap D})\cong \cc_{S_1}\oplus \cc_{S_2}. 
$$
The same is true near every singular point $P_i$. Notice that intersection complexes are constructed locally. Therefore, we have short exact sequence,
$$
0\to \cc_{V_m}\to IC(\cc_{V_m})\to  \bigoplus_{1\leq i\leq \frac{m-1}{2}} \cc_{P_i}\to 0.
$$
Thus, 
$$
\chi((\cc^*)^5, \IC(\cc_{V_m}))-\chi(V_m)=\frac{m-1}{2}.
$$

Denote the torus $(\cc^*)^5$ by $G$. We claim that $CC(\IC(\cc_{V_m}))=[T^*_{V_m}G]$. The characteristic variety is constructed locally. Since $V_m$ has dimension 2, the claim is clearly true along the smooth locus of $V_m$. Near the singular point $P_1$, $\IC(\cc_{V_m\cap D})\cong \cc_{S_1}\oplus \cc_{S_2}$. Therefore, 
$$
CC(\IC(\cc_{V_m\cap D}))=[T^*_{S_1}G]+[T^*_{S_2}G]=[T^*_{V_m\cap D}G]. 
$$
Thus the claim is still true at the singular points of $V_m$. Now, we have
$$
\MLdeg(V_m)=\gdeg(T^*_{V_m}G)=\gdeg(CC(\IC(\cc_{V_m})))=\chi(G, \IC(\cc_{V_m}))=\chi(V_m)+\frac{m-1}{2}
$$
where the first equality is by Lemma \ref{mlg}, the second and the forth are by the above discussion, the third is by Theorem \ref{Riemann}. The equality $\MLdeg(V_m)=\chi(V_m)+\frac{m-1}{2}$ gives what we need to prove in the proposition. 
\end{proof}

\end{document}